\documentclass[a4paper,11pt,notitlepage,reqno]{amsart}

\usepackage[utf8]{inputenc}
\usepackage[T1]{fontenc}
\usepackage{lmodern}
\usepackage[provide=*,main=english,spanish]{babel}
\usepackage{csquotes}

\usepackage[
    margin=28mm,
    marginparwidth=25mm,
    marginparsep=2mm,
    bottom=25mm
]{geometry}
\usepackage[shortlabels]{enumitem}
\usepackage{flafter}
\usepackage{graphicx}

\usepackage{xcolor}
\definecolor{linkInternal}{HTML}{6A1B9A}
\definecolor{linkCitation}{HTML}{00796B}
\definecolor{linkExternal}{HTML}{1565C0}

\usepackage{amssymb}
\usepackage{mathrsfs}
\usepackage{mathtools}

\usepackage[
    backend=biber,
    style=numeric,
    citestyle=numeric-comp,
    sorting=nyt,
    sortcites=true,
    giveninits=true,
    maxbibnames=10,
    minbibnames=1,
    url=false,
    doi=false,
    eprint=false
]{biblatex}

\DeclareNameAlias{author}{given-family}
\DeclareNameAlias{editor}{given-family}
\DeclareNameAlias{translator}{given-family}
\DeclareFieldFormat*{title}{\mkbibemph{#1}}
\DeclareFieldFormat*{journaltitle}{#1}
\DeclareFieldFormat*{booktitle}{#1}
\DeclareFieldFormat[article]{volume}{\mkbibbold{#1}}
\DeclareFieldFormat[article]{number}{\bibstring{number}~#1}
\DeclareFieldFormat[article]{pages}{#1}
\AtBeginBibliography{\setlength{\emergencystretch}{2em}}

\renewbibmacro{in:}{%
  \ifentrytype{article}{}{\printtext{\bibstring{in}\intitlepunct}}}

\renewbibmacro*{journal+issuetitle}{%
  \usebibmacro{journal}%
  \setunit*{\addspace}%
  \iffieldundef{series}{}{\newunit\printfield{series}\setunit{\addspace}}%
  \printfield{volume}%
  \setunit{\addspace}%
  \usebibmacro{issue+date}%
  \setunit{\addcolon\space}%
  \usebibmacro{issue}%
  \setunit{\addcomma\space}%
  \printfield{number}%
  \setunit{\addcomma\space}%
  \printfield{eid}%
  \newunit}

\usepackage[
    pdfencoding=unicode,
    psdextra,
    colorlinks=true,
    linkcolor=linkInternal,
    citecolor=linkCitation,
    urlcolor=linkExternal,
    hypertexnames=false
]{hyperref}
\usepackage{cleveref}

\DeclarePairedDelimiter{\norm}{\lVert}{\rVert}
\DeclarePairedDelimiter{\floor}{\lfloor}{\rfloor}

\DeclareMathOperator{\ran}{Ran}
\DeclareMathOperator{\codim}{codim}

\newcommand{\Real}{\mathfrak{Re}\,}
\newcommand{\Imag}{\mathfrak{Im}\,}

\newcommand{\N}{\mathbb{N}}

\newcommand{\R}{\mathbb{R}}
\newcommand{\C}{\mathbb{C}}
\newcommand{\D}{\mathbb{D}}

\newcommand{\Bloch}{\mathscr{B}}
\newcommand{\BB}{\mathcal{B}}
\newcommand{\LL}{\mathcal{L}}

\newcommand{\TT}{\mathcal{T}}
\newcommand{\EE}{\mathcal{E}}
\newcommand{\CC}{\mathcal{C}}

\newcommand{\Ha}{\mathcal{H}}
\newcommand{\Ah}{\mathcal{A}}

\theoremstyle{plain}
\newtheorem{theorem}{Theorem}[section]
\newtheorem*{theorem*}{Theorem}

\newtheorem{proposition}[theorem]{Proposition}
\newtheorem{lemma}[theorem]{Lemma}

\theoremstyle{definition}

\theoremstyle{remark}
\newtheorem{remark}[theorem]{Remark}

\numberwithin{equation}{section}

\makeatletter
\renewcommand*{\subsection}{%
  \@startsection{subsection}{2}%
    {\z@}%
    {.5\linespacing\@plus.7\linespacing}%
    {.5\linespacing}%
    {\normalfont\itshape}}
\makeatother

\begin{document}

\title[Hausdorff Matrices on the BMOA and Bloch spaces]{Hausdorff Matrices on the BMOA and Bloch spaces}

\author[P. Galanopoulos]{Petros Galanopoulos}
\address{\newline 
    Petros Galanopoulos \newline
    Department of Mathematics, Aristotle University of Thessaloniki,
    54124 Thessaloniki, Greece.}
\email{petrosgala@math.auth.gr}

\author[A. Mahillo]{Alejandro Mahillo}
\address{\newline
    Alejandro Mahillo \newline
    Departamento de Matemáticas, Estadística y Computación,
    Universidad de Cantabria,
    39005 Santander, Spain.}
\email{alejandro.mahillo@unican.es}

\author[J. Oliva-Maza]{Jesus Oliva-Maza}
\address{\newline
    Jesus Oliva-Maza \newline
    Departamento de Matemáticas, Instituto Universitario de Matemáticas y Aplicaciones,
    Universidad de Zaragoza, 50009 Zaragoza, Spain.}
\email{joliva@unizar.es}

\author[A. G. Siskakis]{Aristomenis G. Siskakis}
\address{\newline
    Aristomenis G. Siskakis \newline
    Department of Mathematics, Aristotle University of Thessaloniki,
    54124 Thessaloniki, Greece.}
\email{siskakis@math.auth.gr}

\thanks{A. Mahillo and J. Oliva-Maza have been partially supported by Grant PID2022-137294NBI00 funded by MICIU/AEI/10.13039/501100011033 and by Project E48-20R, Gobierno de Aragón, Spain.}
\thanks{A. Mahillo has also been partially supported by Grant PRE-2020-092311 funded by MICIU/AEI/10.13039/501100011033 and by ``ESF Investing in your future''.}

\begin{abstract}
We determine the spectrum and Fredholm essential spectrum of Hausdorff operators $\Ha_\mu$ on BMOA and the Bloch space $\Bloch$, for suitable finite positive Borel measures $\mu$ on $(0,1]$.
\end{abstract}

\maketitle
\section{Introduction}\label{sec:introduction}

Given a scalar sequence $\{\mu_n\}_{n=0}^\infty$, the induced Hausdorff matrix is the lower triangular matrix $\Ha_{(\mu_n)}=(c_{n,k})_{n,k=0}^\infty$ with entries
\[
    c_{n,k}=\binom{n}{k}\Delta^{n-k}\mu_k, \quad 0\leq k\leq n,
\]
where $\Delta$ is the forward difference operator $\Delta\mu_n=\mu_n-\mu_{n+1}$, $n=0,1,2,\ldots$, and $\Delta^j$ are its iterates:
\[
    \Delta^0=I,\text{ the identity},\qquad\text{ and }\qquad
    \Delta^j=\Delta\circ\Delta^{j-1}, \quad j=1,2,\ldots.
\]
An important special case occurs when
\[
    \mu_n=\int_{(0,1]}t^n\,d\mu(t), \quad n=0,1,2,\ldots,
\]
is the moment sequence of a finite positive Borel measure $\mu$ on $(0,1]$. These matrices are denoted by $\Ha_\mu$ and their entries are found to be
\begin{equation}\label{entries}
    c_{n,k}=\binom{n}{k}\int_{(0,1]}t^k(1-t)^{n-k}\,d\mu(t), \quad 0\leq k\leq n.
\end{equation}

We also consider the transposes $\Ah_\mu=\Ha_\mu^{T}$  called \textit{quasi-Hausdorff matrices}.  The matrices $\Ha_\mu$ and $\Ah_\mu$ were originally studied in connection with the summability theory of series, and later as operators acting on sequence spaces; see \cite{Deddens1978,Hardy1943,Hardy1949,Leibowitz1973,Rhoades1971}. Continuous versions of these operators were also studied on spaces of integrable functions. The recent review article \cite{LiflyandMirotin2024} gives a very good account of work on various versions of Hausdorff operators.

These matrices can also be considered to act on spaces of analytic functions. To fix notation, let $\D$ be the unit disc in the complex plane $\C$, and let $Hol(\D)$ be the space of analytic functions on $\D$. For $f(z)=\sum_{n=0}^\infty a_nz^n\in Hol(\D)$, consider the formal power series obtained by letting the matrices act on the sequence of Taylor coefficients, that is,
\begin{equation}\label{Haus}
    \Ha_\mu(f)(z)=\sum_{n=0}^\infty\left(\sum_{k=0}^n c_{n,k}a_k\right)z^n,
\end{equation}
and
\begin{equation}\label{quasiHaus}
    \Ah_\mu(f)(z)=\sum_{n=0}^\infty\left(\sum_{k=n}^\infty c_{k,n}a_k\right)z^n.
\end{equation}
It is known that the series \eqref{Haus} converges and represents an analytic function on $\D$ for every $f\in Hol(\D)$; see \cite[Lemma 1.1]{GalanopoulosPapadimitrakis2006}. On the other hand, the series \eqref{quasiHaus} is more delicate, since its inner sums need not converge for a general $f\in Hol(\D)$. They are finite when $f$ is a polynomial.

If $X$ is a Banach space consisting of analytic functions on $\D$, one may study boundedness, compactness, and other operator-theoretic properties of $\Ha_\mu$ and $\Ah_\mu$ on $X$ in terms of the inducing measure $\mu$. The first such study was in \cite{GalanopoulosSiskakis2001} on Hardy spaces $H^p$, and then in \cite{GalanopoulosPapadimitrakis2006} on Bergman spaces $A^p$, the Dirichlet space, the Bloch space $\Bloch$, and BMOA, among other spaces.

A key fact in the study of these operators is that they can be represented as averages of weighted composition operators, see \cite{GalanopoulosSiskakis2001,GalanopoulosPapadimitrakis2006}. More precisely,
\begin{equation}\label{eq:integral_S}
    \Ha_\mu(f)(z)=\int_{(0,1]}\frac{f(\Phi_t(z))}{(t-1)z+1}\,d\mu(t), \quad z\in\D,
\end{equation}
for all $f\in Hol(\D)$, where
\[
    \Phi_t(z)=\frac{tz}{(t-1)z+1}, \quad 0<t\leq1,
\]
and, for appropriate \(f \in Hol(\D)\),
\begin{equation}\label{eq:integral_T}
    \Ah_\mu(f)(z) =\int_{(0,1]}f(\Psi_t(z))\,d\mu(t), \quad \Psi_t(z)=tz+1-t.
\end{equation}
Both $\Phi_t$ and $\Psi_t$ are self-maps of $\D$ for $0<t\leq1$, and they induce bounded composition operators on most classical Banach spaces of analytic functions. A computation shows that each family is a multiplicative semigroup under composition:
\begin{equation}\label{mult-semi}
    \Phi_t\circ\Phi_s=\Phi_{ts}, \quad \Psi_t\circ\Psi_s=\Psi_{ts}, \quad 0<t,s\leq1.
\end{equation}
This fact will be used below. These representations, together with known properties of composition operators, make it possible to study $\Ha_\mu$ and $\Ah_\mu$ on various function spaces. We collect some relevant results from \cite{GalanopoulosSiskakis2001,GalanopoulosPapadimitrakis2006} in the following theorem.

\medskip
\noindent\textbf{Theorem A.} \textit{Let $\mu$ be a finite positive Borel measure on $(0,1]$.
\begin{enumerate}[(i)]
    \item If $1<p,q<\infty$ and $1/p+1/q=1$, the following are equivalent: $\Ha_\mu$ is bounded on $H^p$; $\Ah_\mu$ is bounded on $H^q$; and $\int_{(0,1]}t^{-1/q}\,d\mu(t)<\infty$.
    \item The following are equivalent: $\Ha_\mu$ is bounded on $H^1$; $\Ah_\mu$ is bounded on BMOA; $\Ah_\mu$ is bounded on $\Bloch$; and $\int_{(0,1]}\log(1/t)\,d\mu(t)<\infty$.
    \item The following are equivalent: $\Ah_\mu$ is bounded on $H^1$; $\Ha_\mu$ is bounded on BMOA; $\Ha_\mu$ is bounded on $\Bloch$; and $\int_{(0,1]}t^{-1}\,d\mu(t)<\infty$.
\end{enumerate}}
Throughout this paper all measures are positive. Most of the arguments extend, \emph{mutatis mutandis}, to complex measures under the corresponding total-variation assumptions. Since the results above do not provide boundedness characterizations for complex measures, we do not formulate that extension here.

The purpose of this article is to study the spectral properties of $\Ha_\mu$ on BMOA and $\Bloch$, namely the spectrum and the essential spectrum. Recall that, for a closed operator $T$ on a Banach space $X$, its resolvent set and spectrum are
\[
\rho(T)=\{\lambda\in\C:\lambda I-T:D(T)\to X
\text{ is bijective with bounded inverse}\},
\qquad \sigma(T)=\C\setminus\rho(T),
\]
and the essential spectrum is given by
\[
\sigma_{ess}(T)=\{\lambda\in\C:\lambda I-T\text{ is not Fredholm}\}.
\]
Recall also that an operator is Fredholm if its range is closed and its kernel is finite-dimensional and its range has finite codimension. In this case, its index is defined as $\operatorname{Ind}(T)=\dim\ker T-\codim\ran T$.

To reach our goal, we use a modification of a method in a recent paper \cite{AbadiasOliva-Maza2024}, where the spectra and essential spectra of certain generalized Hausdorff matrices are found on a family of function spaces, including Hardy and Bergman spaces. The method in \cite{AbadiasOliva-Maza2024} consists essentially of representing $\Ha_\mu$ as an average of a strongly continuous semigroup and then using functional calculus and spectral mapping theorems to obtain its spectrum.

The semigroups used here are obtained by reparametrizing the multiplicative families in \eqref{mult-semi}. Let $\kappa:(0,1]\to[0,\infty)$ be given by $\kappa(t)=-\log t$, and let $\nu=\kappa(\mu)$ be the pushforward measure, so that $\nu(E)=\mu(\kappa^{-1}(E))$ for every Borel set $E\subseteq[0,\infty)$. The maps
\[
    \phi_t(z)=\frac{e^{-t}z}{(e^{-t}-1)z+1}, \quad
    \psi_t(z)=e^{-t}z+1-e^{-t}, \quad t\geq0,
\]
are additive semigroups of self-maps of $\D$:
\[
    \phi_t\circ\phi_s=\phi_{t+s}, \qquad
    \psi_t\circ\psi_s=\psi_{t+s}, \qquad t,s\geq0.
\]
After the change of variables, the integral representations become
\[
    \Ha_\mu(f)(z)=\int_{[0,\infty)}\frac{f(\phi_t(z))}{(e^{-t}-1)z+1}\,d\nu(t),
    \quad \Ah_\mu(f)(z)=\int_{[0,\infty)}f(\psi_t(z))\,d\nu(t).
\]
The connection between the Cesàro operator and composition semigroups was first exploited in \cite{Cowen1984,Siskakis1987_Cesaro}; see also \cite{Galanopoulos2001,AlemanPersson2010,Persson2008}. The work \cite{AbadiasOliva-Maza2024} extends this approach to more general Hausdorff operators.

However, this method cannot be applied directly on BMOA and $\Bloch$, since these spaces do not support nontrivial strongly continuous semigroups of composition operators; see \cite{BlascoContrerasDiaz-MadrigalMartinezPapadimitrakisSiskakis2013,AndersonJovovicSmith2017,Gallardo-GutierrezSiskakisYakubovich2023}. To overcome this, we use duality. For BMOA we work with $\Ah_\mu$ on $H^1$. For $\Bloch$ we use the matrix operator $\BB_\mu$, which is given initially on polynomials in $A^1$ by
\begin{equation}
	\label{eq:Bmu-definition}
    \BB_\mu(f)(z)=\sum_{k=0}^\infty\left((k+1)\sum_{n=k}^\infty\frac{c_{n,k}}{n+1}a_n\right)z^k,
\end{equation}
and whose dual operator corresponds to \(\Ha_\mu\) on \(\Bloch\)
under the Bloch--Bergman duality.
The crucial observation is that, under suitable assumptions, these operators can be written as functions of the infinitesimal generators \(A_{H^p}\), \(B_{A^p}\) of suitable semigroups,
\[
    \Ah_\mu=\LL\nu(-A_{H^p}), \qquad \BB_\mu=\LL\nu(-B_{A^p}),
\]
where \(\mathcal{L}\nu\) denotes the Laplace transform of \(\nu\), see \Cref{subsec:functional_calculus}.
Functional calculus for sectorial operators transfers their spectral properties to the averages, and duality then transfers those properties to $\Ha_\mu$ on BMOA and $\Bloch$.

Before stating our main results, we introduce some notation appearing in them. We denote the Mellin transform of a finite positive measure $\mu$ on \((0,1]\) by
\[
    \mathcal M\mu(z)=\int_{(0,1]}t^{z-1}\,d\mu(t),
\]
whenever $t^{\Real z-1}$ is integrable with respect to $\mu$. We denote by $\EE$ the extended Dunford--Riesz class, that is, the class of holomorphic functions on some complex sector of angle $\omega\in (\pi/2,\pi)$, which have finite polynomial limits at $0$ and $\infty$ within that sector
(see \Cref{subsec:functional_calculus} for details). 
In the sums below, $\operatorname{ord}_z(\mathcal M\mu-\lambda)$ denotes the order of the zero at $z$, so solutions are counted with multiplicity.

\begin{theorem}\label{th:spectral_result_Hausdorff_BMOA}
Let $\mu$ be a finite positive Borel measure on $(0,1]$ satisfying $\int_{(0,1]}t^{-1}\,d\mu(t)<\infty$. If $\mathcal M\mu\in\EE$, then the spectrum and essential spectrum of $\Ha_\mu$ on BMOA are
\[
    \sigma(\Ha_\mu)=\{\mathcal M\mu(z):\Real z\geq0\}\cup\{\mu(\{1\})\},
\]
\[
    \sigma_{ess}(\Ha_\mu)=\{\mathcal M\mu(z):\Real z=0\}\cup\{\mu(\{1\})\}.
\]
Moreover, for each $\lambda\notin\sigma_{ess}(\Ha_\mu)$, the operator $\lambda I-\Ha_\mu$ is injective and
\[
    \codim\ran(\lambda I-\Ha_\mu)
    =\sum_{\substack{\Real z>0\\\mathcal M\mu(z)=\lambda}}\operatorname{ord}_z(\mathcal M\mu-\lambda)<\infty.
\]
\end{theorem}

\begin{theorem}\label{th:spectral_result_Hausdorff_Bloch}
Let $\mu$ be a finite positive Borel measure on $(0,1]$ satisfying $\int_{(0,1]}t^{-1}\,d\mu(t)<\infty$. If $\mathcal M\mu\in\EE$, then the spectrum and essential spectrum of $\Ha_\mu$ on $\Bloch$ are
\[
    \sigma(\Ha_\mu)=\{\mathcal M\mu(z):\Real z\geq0\}\cup\{\mu(\{1\})\},
\]
\[
    \sigma_{ess}(\Ha_\mu)=\{\mathcal M\mu(z):\Real z=0\}\cup\{\mu(\{1\})\}.
\]
Moreover, for each $\lambda\notin\sigma_{ess}(\Ha_\mu)$, the operator $\lambda I-\Ha_\mu$ is injective and
\[
    \codim\ran(\lambda I-\Ha_\mu)
    =\sum_{\substack{\Real z>0\\\mathcal M\mu(z)=\lambda}}\operatorname{ord}_z(\mathcal M\mu-\lambda)<\infty.
\]
\end{theorem}

The structure of the paper is as follows. \Cref{sec:preliminaries} introduces the function spaces, composition semigroups, and functional calculus used in the proofs. \Cref{sec:proof_theorem_1_1,sec:proof_theorem_1_2} prove the main theorems. In particular, \Cref{sec:proof_theorem_1_2} presents spectral properties of a composition semigroup on the Dirichlet spaces $D^p$. Finally, \Cref{sec:examples} gives examples involving Gamma operators and generalized Cesàro operators.

\section{Preliminaries}\label{sec:preliminaries}

In this section, we recall the function spaces, semigroup results, and sectorial functional calculus used below.

\subsection{Several spaces of analytic functions on \texorpdfstring{$\D$}{D}}\label{subsec:space_analytic_functions}
For $1\leq p<\infty$, the Hardy space $H^p=H^p(\D)$ consists of all analytic functions $f$ on $\D$ such that
\[
    \norm{f}_{H^p}=\sup_{0\leq r<1}\left(\frac{1}{2\pi}\int_0^{2\pi}|f(re^{i\theta})|^p\,d\theta\right)^{1/p}<\infty.
\]
For $1<p<\infty$, the dual space $(H^p)^*$ is anti-linearly isomorphic to $H^q$, where $1/p+1/q=1$, via the pairing
\[
    \langle f,g\rangle=\frac{1}{2\pi}\int_0^{2\pi}\overline{f(e^{i\theta})}g(e^{i\theta})\,d\theta,
    \quad f\in H^q,\ g\in H^p.
\]
More information about Hardy spaces can be found in \cite{Duren1970,Zhu2007}.

For $1\leq p<\infty$, the Bergman space $A^p=A^p(\D)$ consists of all analytic functions on $\D$ such that
\[
    \norm{f}_{A^p}=\left(\int_\D|f(z)|^p\,dA(z)\right)^{1/p}<\infty,
\]
where $dA$ is normalized area measure. Moreover, for $1<p<\infty$, the dual space $(A^p)^*$ is anti-linearly isomorphic to $A^q$, where $1/p+1/q=1$, via
\[
    \langle f,g\rangle=\int_\D\overline{f(z)}g(z)\,dA(z), \quad f\in A^q,\ g\in A^p.
\]
See \cite{DurenSchuster2004,Axler1988} for information about Bergman spaces.

For $1\leq p<\infty$, the Dirichlet space $D^p=D^p(\D)$ consists of all analytic functions whose derivative belongs to $A^p$, with norm
\[
    \norm{f}_{D^p}=\left(|f(0)|^p+\norm{f'}_{A^p}^p\right)^{1/p}.
\]
In both $A^p$ and $D^p$, point evaluations are bounded. In particular,
\begin{equation}\label{eq:point_evaluation_Bergman}
    |f(z)|\leq(1-|z|^2)^{-2/p}\norm{f}_{A^p}, \quad f\in A^p,\ z\in\D,
\end{equation}
see \cite[p. 79]{DurenSchuster2004}. If $f\in D^p$ and $p<2$, then
\begin{equation}\label{eq:point_evaluation_Dirichlet_case_1}
    |f(z)|\lesssim(1-|z|^2)^{1-2/p}\norm{f}_{D^p}, \quad z\in\D,
\end{equation}
which follows by integrating the estimate for $f'$. For $p=2$, we have
\begin{equation}\label{eq:point_evaluation_Dirichlet_case_2}
    |f(z)|\lesssim\left(1+\sqrt{\log\frac{1}{1-|z|^2}}\right)\norm{f}_{D^2}, \quad z\in\D,
\end{equation}
see \cite[Theorem 1.2.1]{El-FallahKellayMashreghiRansford2014}. For $p>2$, $D^p\subset H^\infty(\D)$, where $H^\infty(\D)$ is the space of bounded analytic functions on $\D$, and $\norm{fg}_{D^p}\lesssim\norm{f}_{D^p}\norm{g}_{D^p}$; see \cite[Theorem 4.2]{Wu1999}. In that case,
\begin{equation}\label{eq:point_evaluation_Dirichlet_case_3}
    |f(z)|\lesssim\norm{f}_{D^p}, \quad f\in D^p,\ z\in\D.
\end{equation}

The spaces $D^p$ and $A^p$ are isomorphic under the bounded linear transformation
\begin{equation}\label{eq:isomorphism_A1_D1}
    J:D^p\longrightarrow A^p, \qquad Jf(z)=(zf(z))',
\end{equation}
with inverse
\[
    J^{-1}:A^p\longrightarrow D^p, \qquad
    J^{-1}f(z)=\frac{1}{z}\int_0^z f(\xi)\,d\xi=\int_0^1 f(tz)\,dt.
\]
Indeed,
\[
    \norm{Jf}_{A^p}\leq\norm{f}_{A^p}+\norm{zf'}_{A^p}\lesssim\norm{f}_{D^p},
\]
where we use $\norm{f}_{A^p}\lesssim\norm{f}_{D^p}$, see \cite[Proposition 1.11]{HedenmalmKorenblumZhu2000}. Conversely, radial dilations are contractions on $A^p$, so $\norm{J^{-1}f}_{A^p}\leq\norm{f}_{A^p}$. Since
\[
    (J^{-1}f)'(z)=\frac{f(z)-J^{-1}f(z)}{z},
\]
boundedness of the backward shift on $A^p$ gives $\norm{J^{-1}f}_{D^p}\lesssim\norm{f}_{A^p}$, see \cite[Lemma 4.26]{Zhu2007}.

A function $f\in Hol(\D)$ belongs to BMOA if
\[
    \norm{f}_{BMOA}=|f(0)|+\sup_{a\in\D}\norm{f\circ\varphi_a-f(a)}_{H^2}<\infty,
\]
where $\varphi_a(z)=(a-z)/(1-\overline a z)$. Fefferman's duality theorem states that $(H^1)^*$ is anti-linearly isomorphic to BMOA via the pairing
\[
    \langle f,g\rangle=\lim_{r\to1^-}\frac{1}{2\pi}\int_0^{2\pi}\overline{f(e^{i\theta})}g(re^{i\theta})\,d\theta,
    \quad f\in BMOA,\ g\in H^1.
\]
If $f(z)=\sum_{n=0}^\infty a_nz^n$ and $g(z)=\sum_{n=0}^\infty b_nz^n$, this becomes
\[
    \langle f,g\rangle=\lim_{r\to1^-}\sum_{n=0}^\infty\overline{a_n}b_nr^n.
\]
We refer to \cite{Girela2001} for the theory of BMOA.

Lastly, the Bloch space $\Bloch$ consists of all $f\in Hol(\D)$ such that
\[
    \norm{f}_{\Bloch}=|f(0)|+\sup_{z\in\D}(1-|z|^2)|f'(z)|<\infty.
\]
The space $(A^1)^*$ is anti-linearly isomorphic to $\Bloch$ via the pairing
\[
    \langle f,g\rangle=\lim_{r\to1^-}\int_{r\D}\overline{f(z)}g(z)\,dA(z), \quad f\in\Bloch,\ g\in A^1.
\]
In terms of Taylor coefficients,
\[
    \langle f,g\rangle=\lim_{r\to1^-}\sum_{n=0}^\infty\frac{\overline{a_n}b_n}{n+1}r^{2n+2},
\]
see \cite[Lemma 2.4]{Axler1988} and \cite{Zhu1993}.

\subsection{Strongly continuous semigroups}\label{subsec:composition_semigroups}

A family $(\TT(t))_{t\geq0}$ of bounded operators on a Banach space $X$ is a strongly continuous semigroup if $\TT(0)=I$, $\TT(t+s)=\TT(t)\TT(s)$, and $\TT(t)x\to x$ as $t\to0^+$ for every $x\in X$. Its infinitesimal generator is
\[
    Ax=\lim_{t\to0^+}\frac{\TT(t)x-x}{t},
\]
with domain $D(A)$ consisting of the vectors for which this limit exists. The operator $A$ is closed and densely defined. Its growth bound is
\[
    \omega_0=\lim_{t\to\infty}\frac{\log\norm{\TT(t)}}{t}.
\]
For every $\omega>\omega_0$, there is a constant $M$ such that $\norm{\TT(t)}\leq Me^{\omega t}$ for all $t\geq0$. If $\Real\lambda>\omega_0$, then $\lambda$ is in the resolvent set of $A$, and
\[
    (\lambda I-A)^{-1}x=\int_0^\infty e^{-\lambda t}\TT(t)x\,dt.
\]

Recall the two semiflows
\begin{equation}\label{eq:semigroups}
    \phi_t(z)=\frac{e^{-t}z}{(e^{-t}-1)z+1}, \qquad \psi_t(z)=e^{-t}z+1-e^{-t}.
\end{equation}
The two operator families studied here are
\begin{equation}\label{eq:semigroup_T}
    T(t)f(z) := f(\psi_t(z)), \quad z\in\D,\ t\geq0,
\end{equation}
and
\begin{equation}\label{eq:semigroup_U}
    U(t)f(z):=\frac{d}{dz}\left(\frac{z}{\psi_t(z)}\int_0^{\psi_t(z)}f(\xi)\,d\xi\right), \qquad z \in \D, \, t \geq 0,
\end{equation}
where the quotient is understood by analytic continuation
at the zeros of \(\psi_t\).
The semiflow identity for $\psi_t$ gives the semigroup law for $T(t)$, and a direct calculation gives $U(t+s)=U(t)U(s)$ on $Hol(\D)$. The relation $U(t)=JT(t)J^{-1}$ will only be used between spaces on which $J$ is a bounded isomorphism; see \Cref{remark:similarity}.

\begin{proposition}\label{prop:condition_semigroup_is_strongly}
Let $X$ be a Banach space of analytic functions on $\D$ in which the polynomials are dense, and let $V(t)$ denote either $T(t)$ or $U(t)$. Suppose that each $V(t)$ acts boundedly on $X$ and that
\[
    \sup_{0\leq t<\delta}\norm{V(t)}<\infty
\]
for some $\delta>0$. Then $(V(t))_{t\geq0}$ is a strongly continuous semigroup on $X$.
\end{proposition}
\begin{proof}
By \cite[Proposition I.5.3]{EngelNagel2000}, local uniform boundedness reduces the proof to strong continuity on a dense subset. For each $k\geq0$,
\[
    T(t)(z^k)=\psi_t(z)^k, \qquad
    U(t)(z^k)=\frac{1}{k+1}\frac{d}{dz}\bigl(z\psi_t(z)^k\bigr).
\]
As $t\to0^+$, the coefficients of these polynomials converge to those of $z^k$, while their degrees remain bounded. Since all norms on a finite-dimensional polynomial space are equivalent, $V(t)z^k\to z^k$ in $X$. The same follows for every polynomial, and the result follows by density.
\end{proof}
Under suitable integrability conditions on a positive measure
\(\mu\) on \((0,1]\), we will prove that the operators
\begin{equation}\label{eq:integral_T2}
	T_\mu(f)(z):=\int_{(0,1]}T(-\log t)f(z)\,d\mu(t)
\end{equation}
and
\begin{equation}\label{eq:integral_U}
	U_\mu(f)(z):=\int_{(0,1]}U(-\log t)f(z)\,d\mu(t)
\end{equation}
are bounded on \(H^1\) and \(A^1\), respectively.
Moreover, they agree on polynomials with the matrix operators
\(\Ah_\mu\) and \(\BB_\mu\), defined in~\eqref{quasiHaus}
and~\eqref{eq:Bmu-definition}, respectively.

The following remark records a connection between the two semigroups; see \cite[Sections I.5.10 and II.2.1]{EngelNagel2000}. Recall that in the Introduction we defined the spectrum \(\sigma(A)\) and the essential spectrum \(\sigma_{ess}(A)\) of a closed operator \(A\), and the index \(\operatorname{Ind}(A)\) if \(A\) is Fredholm. We use the symbol \(\sigma_{point}(A)\) to denote the point spectrum of \(A\).
\begin{remark}\label{remark:similarity}
Let $X$ and $Y$ be Banach spaces of analytic functions such that $J:X\to Y$ is an isomorphism. If $(T(t))_{t\geq0}$ is a strongly continuous semigroup on $X$, then $(U(t))_{t\geq0}$ is a similar strongly continuous semigroup on $Y$, with the same growth bound:
\[
    U(t)=JT(t)J^{-1}, \quad t\geq0.
\]
If $A$ is the generator of $T(t)$, then the generator of $U(t)$ is
\[
    B=JAJ^{-1}, \qquad D(B)=J(D(A)).
\]
In particular,
\[
    \sigma(A)=\sigma(B), \quad \sigma_{point}(A)=\sigma_{point}(B), \quad
    \sigma_{ess}(A)=\sigma_{ess}(B).
\]
Indeed, similarity preserves kernels and ranges, and
\[
    \operatorname{Ind}(\lambda I-B)=\operatorname{Ind}(\lambda I-A)
\]
whenever these operators are Fredholm. The resolvents satisfy $(\lambda I-B)^{-1}=J(\lambda I-A)^{-1}J^{-1}$.
\end{remark}

The following results for $T(t)$ follow from the theory of composition semigroups on $H^p$ developed in \cite{BerksonPorta1978}; see also \cite[Example 5.5]{Siskakis1998}. From now on, complex powers of \(1-z\), where \(z\in \D\), are defined using the principal logarithm.
\begin{lemma}\label{lemma:semigroup_T_Hardy}
For $1\leq p<\infty$, $(T(t))_{t\geq0}$ is a strongly continuous semigroup on $H^p$, with $\norm{T(t)}\lesssim e^{t/p}$. Its infinitesimal generator is
\begin{equation}\label{eq:generator_A_Hardy}
    A_{H^p}f(z)=(1-z)f'(z), \quad z\in\D,
\end{equation}
with domain $D(A_{H^p})=\{f\in H^p:(1-z)f'\in H^p\}$.
\end{lemma}

\begin{lemma}\label{lemma:spectrum_A_Hardy}
For $1\leq p<\infty$,
\[
    \sigma_{point}(A_{H^p})=\{\lambda\in\C:\Real\lambda<1/p\}, \qquad
    \sigma(A_{H^p})=\{\lambda\in\C:\Real\lambda\leq1/p\}.
\]
The eigenspace corresponding to $\lambda$ is one-dimensional and generated by $f_\lambda(z)=(1-z)^{-\lambda}$.
\end{lemma}

\begin{remark}
The norm estimate gives $\omega_0\leq1/p$, while the spectral bound satisfies
\[
    1/p=\sup\{\Real\lambda:\lambda\in\sigma(A_{H^p})\}\leq\omega_0
\]
by \cite[Proposition IV.2.2]{EngelNagel2000}. Hence,  $\omega_0=1/p$.
\end{remark}

We also recall the corresponding results on the Bergman spaces; see \cite[Section 4]{Siskakis1987}. We include the generator and its spectral properties for later use.
\begin{lemma}\label{lemma:semigroup_T_Bergman}
For $1\leq p<\infty$, $(T(t))_{t\geq0}$ is a strongly continuous semigroup on $A^p$, with $\norm{T(t)}\leq e^{2t/p}$. Its infinitesimal generator is
\begin{equation}\label{eq:generator_A_Bergman}
    A_{A^p}f(z)=(1-z)f'(z), \quad z\in\D,
\end{equation}
with domain $D(A_{A^p})=\{f\in A^p:(1-z)f'\in A^p\}$.
\end{lemma}

\begin{lemma}\label{lemma:spectrum_A_Bergman}
For $1\leq p<\infty$,
\[
    \sigma_{point}(A_{A^p})=\{\lambda\in\C:\Real\lambda<2/p\}, \qquad
    \sigma(A_{A^p})=\{\lambda\in\C:\Real\lambda\leq2/p\}.
\]
Each eigenspace is one-dimensional and generated by $f_\lambda(z)=(1-z)^{-\lambda}$.
\end{lemma}

\begin{remark}
The norm estimate gives $\omega_0\leq2/p$, while \cite[Proposition IV.2.2]{EngelNagel2000} and \Cref{lemma:spectrum_A_Bergman} give the reverse inequality. Hence, $\omega_0=2/p$.
\end{remark}

\subsection{Functional calculus for sectorial operators}\label{subsec:functional_calculus}

We recall the part of the functional calculus for sectorial operators needed here; see \cite{Haase2006}. Its role is to identify a semigroup average with the Laplace transform of the measure evaluated at the negative generator. Spectral mapping then gives the spectrum of the average.

For each \(\omega\in (0,\pi)\), put 
\[
S_\omega := \{z \in \C \setminus\{0\} \, :\, |\operatorname{Arg} z| < \omega\},
\]
where \(\operatorname{Arg}\) denotes the principal argument in \((-\pi,\pi]\).
A closed operator $A$ on a Banach space \(X\) is sectorial of angle $\frac{\pi}{2}$ if $\sigma(A)\subseteq\overline{S_{\frac{\pi}{2}}}$ and, for every $\frac{\pi}{2}<\omega<\pi$,
\[
    \sup_{\lambda\notin\overline{S_{\omega}}}\norm{\lambda(\lambda I-A)^{-1}}<\infty.
\]
Let \(H_0^\infty(S_\omega)\) denote the class of bounded holomorphic functions \(f\) on \(S_\omega\) that decay polynomially at both \(0\) and \(\infty\), meaning that there exist \(\alpha,\beta>0\) such that
\[
|f(z)|=O(|z|^\alpha)\quad\text{as }z\to0,
\qquad
|f(z)|=O(|z|^{-\beta})\quad\text{as }z\to\infty,
\qquad z\in S_\omega.
\]
Then, the extended Dunford--Riesz class \(\EE\) is defined as
\[
\EE
:=
\bigcup_{\frac{\pi}{2}<\omega<\pi}
\left\{
a+\frac{b}{1+(\cdot)}+h
\,:\,
a,b\in\C,\ h\in H_0^\infty(S_\omega)
\right\},
\]
see \cite[p.~28]{Haase2006}. Equivalently, \(f\in\EE\) if and only if, for some
\(\omega\in(\pi/2,\pi)\), \(f\in H^\infty(S_\omega)\) and \(f\) has finite
polynomial limits at \(0\) and \(\infty\).

Fix \(f\in \EE\) and put \(l_0:= f(0)\), \(l_\infty = f(\infty)\), so
\[
h(z) :=f(z)-l_\infty-\frac{l_0-l_\infty}{1+z} \in H_0^\infty (S_\omega)
\]
for some \(\omega \in (\pi/2,\pi)\). Then \(f(A)\) is defined as
\[
f(A):= 
l_\infty I + (l_0-l_\infty)(I+A)^{-1} + \frac{1}{2\pi i} \int_{\partial S_\varphi} h(z)(z-A)^{-1} \, dz,
\]
where the above integral is independent of the choice of \(\varphi\in (\pi/2,\omega)\), and where \(\partial S_\varphi\) is positively oriented with respect to \(S_\varphi\). It can be shown that \(f(A)\) is a bounded operator on \(X\) and that the mapping from \(\EE\) to the algebra of bounded operators on \(X\) is an algebraic homomorphism, called the \(\EE\)-calculus of \(A\).

We are particularly interested in the spectral mapping theorems of the \(\EE\)-calculus. First, the extended spectrum is defined as $\tilde\sigma(T):=\sigma(T)$ if $T$ is bounded and $\tilde\sigma(T):=\sigma(T)\cup\{\infty\}$ otherwise. Similarly, the extended essential spectrum  is defined as $\tilde\sigma_{ess}(T):=\sigma_{ess}(T)$ if $D(T)$ has finite codimension, and $\tilde\sigma_{ess}(T):=\sigma_{ess}(T)\cup\{\infty\}$ otherwise.
Then the spectral mapping theorems give
\begin{equation}\label{eq:SMT}
    \sigma(f(A))=f(\tilde\sigma(A)), \qquad
    \sigma_{ess}(f(A))=f(\tilde\sigma_{ess}(A)),
\end{equation}
see \cite[Theorem 2.7.8]{Haase2006} and \cite[Theorem 5.4]{Oliva-Maza2023}.

We also use a factorization that makes the Fredholm consequences of \eqref{eq:SMT} explicit. Suppose $\lambda\in \sigma(f(A)) \setminus \sigma_{ess}(f(A))$ and put $M_\lambda=\{z\in\sigma(A):f(z)=\lambda\}$. Then \(M_\lambda\neq \varnothing\) by the spectral mapping theorem and $M_\lambda\cap\sigma_{ess}(A)=\varnothing$. Indeed, a point $z$ in this intersection would give $\lambda=f(z)\in f(\tilde\sigma_{ess}(A))=\sigma_{ess}(f(A))$, contrary to the hypothesis. Hence, $zI-A$ is Fredholm for every $z\in M_\lambda$.

Assume now that $l_0, l_\infty \neq \lambda$. Then $M_\lambda=\{z_1,\ldots,z_N\}$ is finite since the zeros cannot accumulate in $S_\omega$ or at either endpoint. 
Since every $z_j\in\sigma(A)$ lies inside $S_\omega$, the scalar function $f$ is holomorphic on a disc about $z_j$, and we may write $f(z)-\lambda=(z-z_j)^{n_j}h_j(z)$, where $n_j=\operatorname{ord}_{z_j}(f-\lambda)$ and $h_j(z_j)\neq0$.
Define
\[
    g_\lambda(z)=(\lambda-f(z))
    \prod_{j=1}^N\left(\frac{1+z}{z_j-z}\right)^{n_j}.
\]
Then the apparent singularities of \(g_\lambda\) at the points $z_j$ are removable. After filling them in, $g_\lambda$ is bounded and holomorphic on $S_\omega$. Moreover, it is readily seen that \(g_\lambda\) does not vanish on \(\tilde \sigma(A)\) and that it has non-vanishing polynomial limits at \(0\) and \(\infty\).  
Consequently, $g_\lambda\in\EE$ and by the spectral mapping theorem, 
\[0\notin g_\lambda(\tilde\sigma(A)) = \tilde\sigma(g_\lambda(A)),
\]
so \(g_\lambda(A)\) is a bounded and invertible operator.

Rearranging the definition of \(g_\lambda\) gives $\lambda-f(z)=g_\lambda(z)\prod_{j=1}^N((z_j-z)/(1+z))^{n_j}$.
Since $A$ is sectorial, $-1\in\rho(A)$ and then the algebraic properties of the \(\EE\)-calculus yield
\begin{equation}\label{eq:factorization_FC}
    \lambda I-f(A)=g_\lambda(A)
    \prod_{j=1}^N
    \left((z_jI-A)(I+A)^{-1}\right)^{n_j}.
\end{equation}


Set $T_j=(z_jI-A)(I+A)^{-1}$. Since $(I+A)^{-1}:X\to D(A)$ is an isomorphism when $D(A)$ carries the graph norm, $T_j$ and $z_jI-A:D(A)\to X$ have the same kernel dimension, range, Fredholm character, and index. Additivity of the Fredholm index in \eqref{eq:factorization_FC} now yields
\begin{equation}\label{eq:index_FC}
    \operatorname{Ind}(\lambda I-f(A))
    =\sum_{z\in M_\lambda}\operatorname{ord}_z(f-\lambda)
    \operatorname{Ind}(zI-A);
\end{equation}
see \cite[Theorem I.3.16]{Edmunds1987}. Finally, note that if every $zI-A$, $z\in M_\lambda$, is injective, respectively surjective, then the normalized factors \(T_1, \ldots, T_N\) have the same property. Since they commute and $g_\lambda(A)$ is invertible, \eqref{eq:factorization_FC} shows that $\lambda I-f(A)$ is respectively injective or surjective. 

For the point spectrum, let $M=\{0,\infty\}\cap\tilde\sigma(A)$. Then
\begin{equation}\label{eq:SMT_point}
    f(\sigma_{point}(A))\subseteq\sigma_{point}(f(A))
    \subseteq f(\sigma_{point}(A))\cup f(M),
\end{equation}
see \cite[Corollary 6.6]{Haase2005}. 

If $A$ generates a semigroup satisfying $\norm{T(t)}\leq Me^{at}$, then $aI-A$ is sectorial of angle $\pi/2$. Indeed, the semigroup resolvent formula gives
\[
    \norm{((\lambda+a)I-A)^{-1}}\leq\frac{M}{\Real\lambda},
    \qquad \Real\lambda>0.
\]
Together with the elementary geometry of sectors, this is the required sectorial resolvent estimate.
As in the introduction, we are interested in functions that are Laplace transforms of measures. For $a\in\R$, write $f_a(z)=f(z-a)$ and $f\in\EE_a$ when $f_a\in\EE$. We define
\[
    f(-A)=f_a(aI-A).
\]
Let \(\LL\nu\) denote the Laplace transform of a measure \(\nu\) on \([0,\infty)\), that is
\[\LL\nu(z)=\int_{[0,\infty)}e^{-zt}\,d\nu(t),
\]
defined for suitable \(z \in \C\). If $\int_{[0,\infty)}e^{at}\,d\nu(t)<\infty$ and $\LL\nu\in\EE_a$, the semigroup calculus and its compatibility with the sectorial calculus give
\begin{equation}\label{eq:average_calculus}
	\int_{[0,\infty)}T(t)\,d\nu(t)=\LL\nu(-A),
\end{equation}
where the integral is understood in the strong sense, see \cite[Proposition 3.3.2]{Haase2006}. 

Note that if \(\mu\) is a finite measure on \((0,1]\) satisfying \(\nu = \kappa(\mu)\), where \(\kappa:(0,1]\to [0,\infty)\) is given by \(\kappa(t) = -\log t\), then we have
\[
\mathcal M\mu(z)=\LL\nu(z-1), \qquad \Real z>1.
\]
Hence, $\LL\nu\in\EE_a$ is equivalent to $\mathcal M\mu\in\EE_{a-1}$, and the Mellin transform allows us to state the results directly in terms of the original measure \(\mu\) on $(0,1]$. Finally, note that whenever $\mathcal M\mu \in \EE_a$ for some \(a\in \R\), then
\[
\mathcal M\mu(\infty)=\lim_{x\to\infty}\int_{(0,1]}t^{x-1}\,d\mu(t)=\mu(\{1\}).
\]

\section{Proof of Theorem \ref{th:spectral_result_Hausdorff_BMOA}}\label{sec:proof_theorem_1_1}

We first identify $\Ah_\mu$ with the integral operator $T_\mu$ on $H^p$ defined in~\eqref{eq:integral_T2}, compute its spectral properties, and transfer them to $\Ha_\mu$ on BMOA by duality.

\begin{lemma}\label{lemma:A_equals_T}
Let $1\leq p<\infty$ and let $\mu$ be a finite positive Borel measure on $(0,1]$ such that $\int_{(0,1]}t^{-1/p}\,d\mu(t)<\infty$. Then $T_\mu$ is bounded on $H^p$ and $\Ah_\mu=T_\mu$.
\end{lemma}
\begin{proof}
The boundedness of $\Ah_\mu$ is given in \cite[Theorem 2.1]{GalanopoulosPapadimitrakis2006}. On the other hand, by \Cref{lemma:semigroup_T_Hardy},
\[
    \norm{T_\mu f}_{H^p}\leq\int_{[0,\infty)}\norm{T(t)f}_{H^p}\,d\nu(t)\lesssim\norm{f}_{H^p}\int_{[0,\infty)}e^{t/p}\,d\nu(t)=\norm{f}_{H^p}\int_{(0,1]}t^{-1/p}\,d\mu(t).
\]
Thus, both operators are bounded. They agree on polynomials, so density gives $\Ah_\mu=T_\mu$ on $H^p$.
\end{proof}

To determine the essential spectrum of $A_{H^p}$, we first prove two range lemmas.

\begin{remark}\label{remark:range_A}
Let $f\in H^p$ and $\lambda\in\C$. Then $f\in\ran(\lambda I-A_{H^p})$ if and only if there exists $K\in\C$ such that
\begin{equation}\label{eq:g_surjectivity}
    g(z)=\frac{1}{(1-z)^\lambda}\left(K-\int_0^z(1-\xi)^{\lambda-1}f(\xi)\,d\xi\right)
\end{equation}
belongs to $H^p$. Indeed, in this case $g\in D(A_{H^p})$ and $(\lambda I-A_{H^p})g=f$.
\end{remark}

If $\Real\lambda>1/p$, surjectivity follows from \Cref{lemma:spectrum_A_Hardy}. If $\Real\lambda<1/p$, then $(1-z)^{-\lambda}\in H^p$, so the choice of $K$ does not affect membership in $H^p$. For suitable functions $f$, we consider
\[
    (\Lambda_{0,\lambda}f)(z)=-\frac{1}{(1-z)^\lambda}\int_0^z(1-\xi)^{\lambda-1}f(\xi)\,d\xi,
\]
\[
    (\Lambda_{1,\lambda}f)(z)=\frac{1}{(1-z)^\lambda}\int_z^1(1-\xi)^{\lambda-1}f(\xi)\,d\xi,
\]
where the last integral is taken along a path from \(z\) to \(1\), contained in \(\D\) except for the endpoint \(1\), and approaching \(1\) non-tangentially. Whenever we use \(\Lambda_{1,\lambda}f\), the value of this integral will be independent of the choice of such a path.

\begin{lemma}\label{lemma:ess_spec1}
Let $1\leq p<\infty$, $\delta>0$, and $\Real\lambda>1/p-\delta$. If $f\in H^p$ satisfies $f/(1-(\cdot))^\delta\in H^p$, then $g=\Lambda_{1,\lambda}f$ belongs to $H^p$ and $(\lambda I-A_{H^p})g=f$.
\end{lemma}
\begin{proof}
First, we prove that
\[
    \tilde g=\int_0^\infty e^{-\lambda t}T(t)f\,dt
\]
is a Bochner integral in $H^p$. Since $1-\psi_t(z)=e^{-t}(1-z)$,
\[
    T(t)f=(1-(\cdot))^\delta e^{-\delta t}
    T(t)\left(\frac{f}{(1-(\cdot))^\delta}\right),
\]
and consequently
\[
    \norm{T(t)f}_{H^p}\lesssim e^{(1/p-\delta)t}
    \norm*{\frac{f}{(1-(\cdot))^\delta}}_{H^p}.
\]
Thus, the integral converges because $\Real\lambda>1/p-\delta$. The change of variable $\xi=\psi_t(z)$ gives
\begin{equation}\label{eq:lambda_change_variable}
    \tilde g(z)=\frac{1}{(1-z)^\lambda}\int_z^1(1-\xi)^{\lambda-1}f(\xi)\,d\xi
    =(\Lambda_{1,\lambda}f)(z),
\end{equation}
where the integration path is $\{\psi_t(z):t\geq0\}$. Since the Hardy-space pointwise estimate gives
\[
|f(z)|\lesssim\frac{|1-z|^\delta}{(1-|z|)^{1/p}}
\norm*{\frac{f}{(1-(\cdot))^\delta}}_{H^p},
\]
it follows from Cauchy's theorem and the condition \(\Real\lambda+\delta>1/p\) that the integral may equivalently be taken along any path approaching \(1\) non-tangentially. Therefore, setting \(K:=\tilde g(0)\), we obtain
\[
\frac{1}{(1-z)^\lambda}\left(K-\int_0^z(1-\xi)^{\lambda-1}f(\xi)\,d\xi\right) = (\Lambda_{1,\lambda}f)(z).
\]
The conclusion now follows from \Cref{remark:range_A}.
\end{proof}


\begin{lemma}\label{lemma:ess_spec2}
Let $1\leq p<\infty$, $\Real\lambda<1/p$, and $m\in\N_0$ with $m>2/p-\Real\lambda$. If $f\in H^p$ has a zero of order at least $m$ at $0$, then $g=\Lambda_{0,\lambda}f$ belongs to $H^p$ and $(\lambda I-A_{H^p})g=f$.
\end{lemma}
\begin{proof}
The integral defining $\Lambda_{0,\lambda}f$ converges for every $z\in\D$. Using $\xi=\phi_t(z)$, we obtain
\begin{align*}
    -(\Lambda_{0,\lambda}f)(z)
    &=\frac{1}{(1-z)^\lambda}\int_0^z(1-\xi)^{\lambda-1}f(\xi)\,d\xi\\
    &=\int_0^\infty\phi_t(z)\left(\frac{1-\phi_t(z)}{1-z}\right)^\lambda f(\phi_t(z))\,dt.
\end{align*}
The estimate used here is a particular case of \cite[Proposition 3.1]{AbadiasOliva-Maza2024}:
\begin{equation}\label{eq:growth_bound_phit}
    \norm*{\phi_t\left(\frac{1-\phi_t}{1-(\cdot)}\right)^\lambda(f\circ\phi_t)}_{H^p}
    \lesssim e^{(\Real\lambda-1/p)t}\norm{f}_{H^p}.
\end{equation}
For completeness, it follows from the following weighted-composition estimate. The image of the unit circle under $\phi_t$ is a circle in $\overline\D$, and its arc-length measure has a Carleson bound independent of $t$. Hence,
\[
    \norm{(\phi_t')^{1/p}(h\circ\phi_t)}_{H^p}\lesssim\norm{h}_{H^p},
    \quad h\in H^p,\ t\geq0.
\]
On the other hand,
\[
    \phi_t(z)^{m+1}(\phi_t'(z))^{-1/p}
    \left(\frac{1-\phi_t(z)}{1-z}\right)^\lambda
    =\frac{e^{-(m+1-1/p)t}z^{m+1}}
    {((e^{-t}-1)z+1)^{m+\lambda+1-2/p}}.
\]
Since $m+\Real\lambda+1-2/p>0$, the supremum norm of this expression is at most
\[
    e^{(\Real\lambda-1/p)t+|\Imag\lambda|\pi/2}.
\]
Apply the weighted-composition estimate to $h(z)=f(z)/z^m$, which belongs to $H^p$ and has the same norm as $f$. This gives \eqref{eq:growth_bound_phit}. The integrand is strongly continuous, as follows first on polynomials and then by approximation using the same local bounds. Since $\Real\lambda<1/p$, its integral converges in $H^p$. Now apply \Cref{remark:range_A}.
\end{proof}

\begin{lemma}\label{lemma:ess_spectrum_A1}
Let $1\leq p<\infty$. If $\Real\lambda<1/p$, then $\lambda I-A_{H^p}$ is surjective and
\[
    \dim\ker(\lambda I-A_{H^p})=\operatorname{Ind}(\lambda I-A_{H^p})=1.
\]
Moreover,
\[
    \tilde\sigma_{ess}(A_{H^p})=\{\lambda\in\C:\Real\lambda=1/p\}\cup\{\infty\}.
\]
\end{lemma}
\begin{proof}
Fix $\Real\lambda<1/p$ and $f\in H^p$. Choose $n\in\N$ sufficiently large that
\[
    \Real\lambda>1/p-\floor{n/2}, \qquad \floor{n/2}>2/p-\Real\lambda.
\]
Write
\begin{equation}\label{eq:f_decomposition}
    f=\sum_{k=0}^n\binom{n}{k}f_k, \qquad f_k(z)=z^k(1-z)^{n-k}f(z).
\end{equation}
This is the identity $1=(z+(1-z))^n$, arranged so that the two preceding lemmas apply to complementary terms.
If $k\leq\floor{n/2}$, then $f_k/(1-(\cdot))^{\floor{n/2}}\in H^p$, so $f_k\in\ran(\lambda I-A_{H^p})$ by \Cref{lemma:ess_spec1}. If $k>\floor{n/2}$, the same conclusion follows from \Cref{lemma:ess_spec2}. Thus, $\lambda I-A_{H^p}$ is surjective. Its kernel is one-dimensional by \Cref{lemma:spectrum_A_Hardy}, which gives index $1$.

For $\Real\lambda>1/p$, the operator is invertible and has index $0$. Since the Fredholm set is open and the index is locally constant, every point on $\Real\lambda=1/p$ belongs to the essential spectrum. There are no other finite points in the essential spectrum by what we have proved. Finally, the extended essential spectrum is closed in the Riemann sphere, so it also contains $\infty$.
\end{proof}

The generator of the uniformly bounded semigroup $(e^{-t/p}T(t))_{t\geq0}$ is $A_{H^p}-1/p$, so $1/p-A_{H^p}$ is sectorial of angle $\pi/2$. We now apply the functional calculus to obtain the spectral properties of $\Ah_\mu$.

\begin{lemma}\label{lemma:spectrum_T}
Let $1\leq p<\infty$ and let $\mu$ be a finite positive Borel measure on $(0,1]$ satisfying $\int_{(0,1]}t^{-1/p}\,d\mu(t)<\infty$. If $\mathcal M\mu\in\EE_{1/p-1}$, then $\Ah_\mu$ on $H^p$ satisfies:
\begin{enumerate}[(i)]
    \item $\displaystyle\sigma(\Ah_\mu)=\{\mathcal M\mu(z):\Real z\geq1-1/p\}\cup\{\mu(\{1\})\}$.
    \item $\displaystyle\sigma_{ess}(\Ah_\mu)=\{\mathcal M\mu(z):\Real z=1-1/p\}\cup\{\mu(\{1\})\}$.
    \item If $\lambda\notin\sigma_{ess}(\Ah_\mu)$, then $\lambda I-\Ah_\mu$ is surjective and
    \[
        \dim\ker(\lambda I-\Ah_\mu)=
        \sum_{\substack{\Real z>1-1/p\\\mathcal M\mu(z)=\lambda}}
        \operatorname{ord}_z(\mathcal M\mu-\lambda)<\infty.
    \]
    \item The point spectrum satisfies
    \begin{align*}
        \sigma_{point}(\Ah_\mu)
        &\subseteq\{\mathcal M\mu(z):\Real z>1-1/p\}\cup\{\mathcal M\mu(1-1/p),\mu(\{1\})\},\\
        \sigma_{point}(\Ah_\mu)
        &\supseteq\{\mathcal M\mu(z):\Real z>1-1/p\}.
    \end{align*}
\end{enumerate}
\end{lemma}
\begin{proof}
Let $\nu=\kappa(\mu)$, so $\mathcal M\mu(z)=\LL\nu(z-1)$. By \eqref{eq:average_calculus} and \Cref{lemma:A_equals_T},
\[
    \Ah_\mu = T_\mu=\int_{[0,\infty)}T(t)\,d\nu(t)=\LL\nu(-A_{H^p}).
\]
The spectrum of $1/p-A_{H^p}$ is $\overline{S_{\pi/2}}$, its point spectrum is $S_{\pi/2}$, and its extended essential spectrum is $i\R\cup\{\infty\}$, by \Cref{lemma:spectrum_A_Hardy,lemma:ess_spectrum_A1}. Thus, (i) and (ii) follow from \eqref{eq:SMT}, including the value $\LL\nu(\infty)=\mu(\{1\})$. For $\lambda\notin\sigma_{ess}(\Ah_\mu)$, the endpoint values are excluded and the zeros in the factorization are finite. The factorization and \eqref{eq:index_FC}, together with surjectivity in \Cref{lemma:ess_spectrum_A1}, then give (iii). Finally, \eqref{eq:SMT_point} gives (iv).
\end{proof}

We finish with the duality argument. Although the pairings are anti-linear in their first variable, the entries $c_{n,k}$ are real because $\mu$ is positive.

\begin{lemma}\label{lemma:duality_BMOA}
Let $\mu$ be a finite positive Borel measure on $(0,1]$ such that $\int_{(0,1]}t^{-1}\,d\mu(t)<\infty$. Then $\Ha_\mu$ is bounded on BMOA and
\[
    \langle\Ha_\mu f,g\rangle=\langle f,\Ah_\mu g\rangle,
    \quad f\in BMOA,\ g\in H^1.
\]
\end{lemma}
\begin{proof}
Boundedness follows from \cite[Theorems 2.1 and 4.2]{GalanopoulosPapadimitrakis2006}. Let $f(z)=\sum a_kz^k\in BMOA$ and let $g(z)=\sum_{n=0}^d b_nz^n$ be a polynomial. The coefficient representation of the pairing gives
\[
    \langle\Ha_\mu f,g\rangle=\sum_{n=0}^d\sum_{k=0}^n c_{n,k}\overline{a_k}b_n=\sum_{k=0}^d\overline{a_k}\sum_{n=k}^d c_{n,k}b_n=\langle f,\Ah_\mu g\rangle.
\]
All sums are finite. Density of polynomials in $H^1$ now extends the equality to every $g\in H^1$.
\end{proof}

Under the anti-linear isomorphism between BMOA and $(H^1)^*$, this lemma identifies $\Ha_\mu$ with the adjoint of $\Ah_\mu$. Anti-linear duality initially conjugates the spectral parameter. Since the entries of $\Ah_\mu$ are real, coefficientwise conjugation commutes with $\Ah_\mu$ and interchanges the parameters $\lambda$ and $\bar\lambda$. This removes the conjugation and gives
\[
    \sigma(\Ha_\mu)=\sigma(\Ah_\mu) 
    , \qquad
    \sigma_{ess}(\Ha_\mu)=\sigma_{ess}(\Ah_\mu)
    .
\]
For $\lambda$ outside the essential spectrum, adjoint duality also gives
\[
    \dim\ker(\lambda I-\Ha_\mu)=\codim\ran(\lambda I-\Ah_\mu), \qquad
    \codim\ran(\lambda I-\Ha_\mu)=\dim\ker(\lambda I-\Ah_\mu),
\]
see \cite[p. 205]{Yosida1965}. Now \Cref{lemma:spectrum_T} with $p=1$ proves \Cref{th:spectral_result_Hausdorff_BMOA}.

\section{Proof of Theorem \ref{th:spectral_result_Hausdorff_Bloch}}\label{sec:proof_theorem_1_2}

We first study $(T(t))_{t\geq0}$ on $D^p$ and use similarity to obtain the spectral properties of $(U(t))_{t\geq0}$ on $A^p$. We then identify $\BB_\mu$ with $U_\mu$ and finish the proof by duality between $A^1$ and $\Bloch$.

\begin{lemma}\label{lemma:semigroup_T_Dirichlet}
For $1\leq p<\infty$, $(T(t))_{t\geq0}$ is a strongly continuous semigroup on $D^p$. Its growth bound satisfies $\omega_0\leq2/p-1$ for $p<2$ and $\omega_0=0$ for $p\geq2$. Its infinitesimal generator is
\begin{equation}\label{eq:generator_A_Dirichlet}
    A_{D^p}f(z)=(1-z)f'(z), \quad z\in\D,
\end{equation}
with domain $D(A_{D^p})=\{f\in D^p:(1-z)f'\in D^p\}$.
\end{lemma}
\begin{proof}
We divide the proof into three cases, using the pointwise estimates in \Cref{subsec:space_analytic_functions}. Since $\psi_t(0)=1-e^{-t}$, they give
\[
    |f(\psi_t(0))|\lesssim
    \begin{cases}
        e^{(2/p-1)t}\norm{f}_{D^p},&p<2,\\
        (1+\sqrt t)\norm{f}_{D^2},&p=2,\\
        \norm{f}_{D^p},&p>2.
    \end{cases}
\]
Also, a change of variable in the area integral gives
\[
    \norm{(T(t)f)'}_{A^p}=e^{-t}\norm{f'\circ\psi_t}_{A^p}
    \leq e^{(2/p-1)t}\norm{f'}_{A^p}.
\]
Consequently,
\begin{equation}\label{eq:bound_semigroup_Ap}
    \norm{T(t)}_{D^p\to D^p}\lesssim
    \begin{cases}
        e^{(2/p-1)t},&p<2,\\
        1+\sqrt t,&p=2,\\
        1,&p>2.
    \end{cases}
\end{equation}
Polynomials are dense in $D^p$, so \Cref{prop:condition_semigroup_is_strongly} gives strong continuity. The estimates imply
\begin{equation}\label{eq:growth_bound}
    \omega_0\leq\max\{2/p-1,0\}.
\end{equation}
Since $T(t)$ fixes the constants, $\omega_0\geq0$, giving equality when $p\geq2$. Finally, the description of the generator follows from \cite[Theorem 2.1]{Gallardo-GutierrezSiskakisYakubovich2023}.
\end{proof}

\begin{lemma}\label{lemma:point_spectrum_A_Dirichlet}
For $1\leq p<\infty$,
\[
    \sigma_{point}(A_{D^p})=\{\lambda\in\C:\Real\lambda<2/p-1\}\cup\{0\}.
\]
Each eigenspace is one-dimensional and generated by $f_\lambda(z)=(1-z)^{-\lambda}$.
\end{lemma}
\begin{proof}
Solving $(1-z)f'(z)=\lambda f(z)$ gives $f(z)=C(1-z)^{-\lambda}$. For $\lambda=0$ this is a constant and belongs to $D^p$. For $\lambda\neq0$, it belongs to $D^p$ if and only if $(1-z)^{-\lambda-1}\in A^p$, that is, $\Real\lambda<2/p-1$, see \cite[Section 3.2]{DurenSchuster2004}.
\end{proof}

If $p>2$, integrating \eqref{eq:point_evaluation_Bergman} along radii shows that every $f\in D^p$ has a finite radial limit at $1$. In particular, $f\mapsto f(1)$ is a bounded functional on $D^p$.

\begin{lemma}\label{lemma:range_A_Dirichlet}
Let $1\leq p<\infty$. Then:
\begin{enumerate}[(i)]
    \item $\ran(\lambda I-A_{D^p})=D^p$ if $\Real\lambda<2/p-1$, or if $\Real\lambda>2/p-1$ and $\lambda\neq0$;
    \item $\ran(A_{D^p})=\{f\in D^p:f(1)=0\}$ if $p>2$;
    \item $\ran(\lambda I-A_{D^p})$ is not closed if $\Real\lambda=2/p-1$.
\end{enumerate}
\end{lemma}
\begin{proof}
Assume first that $\Real\lambda<2/p-1$. The description of the range in \Cref{remark:range_A} applies equally to $D^p$. Choose an integer $\delta>\max\{2/p-1,0\}-\Real\lambda$. Repeating the proof of \Cref{lemma:ess_spec1}, using \eqref{eq:bound_semigroup_Ap}, gives
\[
    \left\{f\in D^p:\frac{f}{(1-(\cdot))^\delta}\in D^p\right\}
    \subseteq\ran(\lambda I-A_{D^p}).
\]
In the case $p=2$, the additional factor $1+\sqrt t$ is integrable against the resulting decaying exponential.

Next, let $m\in\N$ with $m>2/p-\Real\lambda$ and take $f\in D^p$ with a zero of order at least $m$ at $0$. As before,
\[
    -(\Lambda_{0,\lambda}f)(z)=\int_0^\infty
    \phi_t(z)\left(\frac{1-\phi_t(z)}{1-z}\right)^\lambda f(\phi_t(z))\,dt.
\]
The auxiliary estimate in \Cref{lemma:dirichlet_growth_bound} gives exponential decay with exponent $\Real\lambda+1-2/p<0$. The integrand is strongly continuous by polynomial approximation and the same local bounds. Hence, $\Lambda_{0,\lambda}f\in D^p$, and $f\in\ran(\lambda I-A_{D^p})$. Finally, the polynomial decomposition \eqref{eq:f_decomposition}, with $n$ sufficiently large, writes every $f\in D^p$ as a sum of functions of these two types. This proves surjectivity in the left half-plane.

For the opposite half-plane, let $D_0^p=\{f\in D^p:f(0)=0\}$, which is isometrically isomorphic to the quotient $D^p/\C$. Since constants are invariant, the quotient semigroup is represented on $D_0^p$ by
\[
    T_0(t)f=T(t)f-(T(t)f)(0).
\]
Let $A_{D_0^p}$ denote the generator of $(T_0(t))_{t\geq0}$ on $D_0^p$.
Differentiation is an isometric isomorphism from $D_0^p$ onto $A^p$, and
\[
    (T_0(t)f)'=e^{-t}T(t)f'.
\]
Thus, $A_{D_0^p}$ is similar to $A_{A^p}-I$. By \Cref{lemma:spectrum_A_Bergman}, $\lambda I-A_{D_0^p}$ is invertible when $\Real\lambda>2/p-1$. Since $A_{D^p}$ vanishes on constants, $\lambda I-A_{D^p}$ is then invertible for $\lambda\neq0$. If $p>2$ and $\lambda=0$, its kernel is $\C$ and its range has codimension one; see also the quotient construction in \cite[Section II.2.4]{EngelNagel2000}.

For $p>2$ and $g\in D(A_{D^p})$, estimate \eqref{eq:point_evaluation_Bergman} gives
\[
    |(1-r)g'(r)|\lesssim(1-r)^{1-2/p}\norm{g'}_{A^p}\longrightarrow0.
\]
Therefore, $\ran(A_{D^p})\subseteq\{f:f(1)=0\}$. Both spaces have codimension one, proving (ii).

Finally, each point on $\Real\lambda=2/p-1$ is a boundary point of the spectrum and Hence,  an approximate eigenvalue. If $\lambda\neq0$, it is not an eigenvalue by \Cref{lemma:point_spectrum_A_Dirichlet}, so the range is not closed. When $p=2$ and $\lambda=0$, the similarity above and \Cref{lemma:spectrum_A_Bergman} show that $0$ is a boundary point, but not an eigenvalue, of $A_{D_0^2}$. Thus, its range is not closed. Its inverse image under the quotient map is $\C+\ran(A_{D^2})$, so $\ran(A_{D^2})$ cannot be closed either.
\end{proof}

\begin{remark}\label{rem:surjectivity_A}
The resolvent in the case $p>2$ also has a useful explicit expression. If $\Real\lambda>0$, the semigroup formula gives
\[
    ((\lambda I-A_{D^p})^{-1}f)(z)=\int_0^\infty e^{-\lambda t}(T(t)f)(z)\,dt=\frac{1}{(1-z)^\lambda}\int_z^1(1-\xi)^{\lambda-1}f(\xi)\,d\xi.
\]
For $\Real\lambda>2/p-1$ and $\lambda\neq0$, integration by parts suggests the expression
\[
    g_\lambda(z)=\frac{f(z)}{\lambda}
    +\frac{1}{\lambda(1-z)^\lambda}\int_z^1 f'(\xi)(1-\xi)^\lambda\,d\xi.
\]
It remains valid when $2/p-1<\Real\lambda\leq0$. Indeed, along a non-tangential path to $1$,
\[
    |f'(\xi)(1-\xi)^\lambda|\lesssim
    |1-\xi|^{\Real\lambda-2/p}\norm{f'}_{A^p},
\]
so the integral converges. Differentiating gives
\[
    g_\lambda'(z)=\frac{1}{(1-z)^{\lambda+1}}
    \int_z^1f'(\xi)(1-\xi)^\lambda\,d\xi
    =\int_0^\infty e^{-(\lambda+1)t}f'(\psi_t(z))\,dt.
\]
Consequently,
\[
    \norm{g_\lambda'}_{A^p}\leq\norm{f'}_{A^p}
    \int_0^\infty e^{(2/p-\Real\lambda-1)t}\,dt<\infty.
\]
Thus, $g_\lambda\in D^p$ and $\lambda g_\lambda-(1-z)g_\lambda'=f$, proving the resolvent formula directly.
\end{remark}

\begin{lemma}\label{lemma:spectrum_A_Dirichlet}
For $1\leq p<\infty$, the extended spectrum and extended essential spectrum of $A_{D^p}$ are
\[
    \tilde\sigma(A_{D^p})=\{\lambda\in\C:\Real\lambda\leq2/p-1\}\cup\{0,\infty\},
\]
\[
    \tilde\sigma_{ess}(A_{D^p})=\{\lambda\in\C:\Real\lambda=2/p-1\}\cup\{\infty\}.
\]
Moreover,
\[
    \operatorname{Ind}(\lambda I-A_{D^p})=
    \begin{cases}
        1,&\Real\lambda<2/p-1,\\
        0,&\Real\lambda>2/p-1.
    \end{cases}
\]
\end{lemma}
\begin{proof}
The finite-plane spectral sets and the index follow from \Cref{lemma:point_spectrum_A_Dirichlet,lemma:range_A_Dirichlet}. The extended sets contain $\infty$ because their finite parts are unbounded and the extended sets are closed in the Riemann sphere.
\end{proof}

\begin{remark}
For $1\leq p<2$, the preceding lemma shows that the growth bound of $(T(t))_{t\geq0}$ on $D^p$ is $\omega_0=2/p-1$, improving the inequality in \Cref{lemma:semigroup_T_Dirichlet}.
\end{remark}

The corresponding results for $U(t)$ on $A^p$ follow from \Cref{remark:similarity} and the isomorphism \eqref{eq:isomorphism_A1_D1}.

\begin{lemma}\label{lemma:semigroup_U_Bergman}
For $1\leq p<\infty$, $(U(t))_{t\geq0}$ is a strongly continuous semigroup on $A^p$, with growth bound $\omega_0=\max\{2/p-1,0\}$. Its generator is
\begin{equation}\label{eq:generator_U_Bergman}
    B_{A^p}f(z)=\frac{1}{z^2}\int_0^z f(\xi)\,d\xi-\frac{f(z)}{z}+(1-z)f'(z),
\end{equation}
with domain $D(B_{A^p})=\{f\in A^p:B_{A^p}f\in A^p\}$. The singularity at $0$ is removable.
\end{lemma}

\begin{lemma}\label{lemma:spectrum_U_Bergman}
For $1\leq p<\infty$,
\[
    \sigma_{point}(B_{A^p})=\{\lambda\in\C:\Real\lambda<2/p-1\}\cup\{0\}.
\]
Each eigenspace is one-dimensional and generated by
\[
    g_\lambda(z)=\frac{1+(\lambda-1)z}{(1-z)^{\lambda+1}}.
\]
Also, $\lambda I-B_{A^p}$ is surjective if $\Real\lambda<2/p-1$, or if $\Real\lambda>2/p-1$ and $\lambda\neq0$. If $p>2$, then
\[
    \ran(B_{A^p})=\left\{f\in A^p:\int_0^1f(\xi)\,d\xi=0\right\}.
\]
The range of $\lambda I-B_{A^p}$ is not closed on $\Real\lambda=2/p-1$. Consequently,
\[
    \tilde\sigma(B_{A^p})=\{\lambda\in\C:\Real\lambda\leq2/p-1\}\cup\{0,\infty\},
\]
\[
    \tilde\sigma_{ess}(B_{A^p})=\{\lambda\in\C:\Real\lambda=2/p-1\}\cup\{\infty\},
\]
and
\[
    \operatorname{Ind}(\lambda I-B_{A^p})=
    \begin{cases}
        1,&\Real\lambda<2/p-1,\\
        0,&\Real\lambda>2/p-1.
    \end{cases}
\]
\end{lemma}

We now identify the matrix operator $\BB_\mu$ with the integral operator $U_\mu$, defined in~\eqref{eq:Bmu-definition} and~\eqref{eq:integral_U}, respectively, beginning with polynomials.

\begin{lemma}\label{lemma:B_equals_U_on_polynomials}
Let $\mu$ be a finite positive Borel measure on $(0,1]$. If $f$ is a polynomial, then $\BB_\mu f=U_\mu f$.
\end{lemma}
\begin{proof}
Integrating $U(t)=JT(t)J^{-1}$ against $\nu$ gives $U_\mu=JT_\mu J^{-1}$ on polynomials. Since $J(z^k)=(k+1)z^k$, $J^{-1}(z^n)=z^n/(n+1)$, and
\[
T_\mu(z^n)=\int_{(0,1]}(tz+1-t)^n\,d\mu(t)=\sum_{k=0}^n c_{n,k}z^k,
\]
it follows that
\[
U_\mu(z^n) = J T_\mu J^{-1} (z^n) = \frac{1}{n+1} \sum_{k=0}^n c_{n,k} J(z^k) = \frac{1}{n+1} \sum_{k=0}^n (k+1) c_{n,k}z^k =\BB_\mu(z^n).
\]
Then the result follows by linearity.
\end{proof}

\begin{lemma}\label{lemma:B_equals_U}
Let $\mu$ be a finite positive Borel measure on $(0,1]$. For $1\leq p<\infty$, $p\neq2$, the operator $U_\mu$ is bounded on $A^p$ if
\[
    \int_{(0,1]}t^{1-2/p}\,d\mu(t)<\infty.
\]
For $p=2$, a sufficient condition is
\[
    \int_{(0,1]}\left(1+\sqrt{\log(1/t)}\right)d\mu(t)<\infty.
\]
In these cases $\BB_\mu$ extends uniquely from polynomials to a bounded operator on $A^p$, and $\BB_\mu=U_\mu$. For $p>2$, every finite positive Borel measure satisfies the condition.
\end{lemma}
\begin{proof}
By similarity, the estimates \eqref{eq:bound_semigroup_Ap} hold for $U(t)$ on $A^p$ as well. Thus,
\[
    \norm{U_\mu f}_{A^p}\leq\int_{[0,\infty)}\norm{U(t)f}_{A^p}\,d\nu(t)<\infty
\]
under the stated conditions. Equality with $\BB_\mu$ holds on polynomials by the preceding lemma. Their density in $A^p$ gives the unique bounded extension and the equality on $A^p$.
\end{proof}

Put $\gamma_p=\max\{2/p-1,0\}$. If $p\neq2$, then $(e^{-\gamma_pt}U(t))_{t\geq0}$ is uniformly bounded, so $\gamma_p-B_{A^p}$ is sectorial of angle $\pi/2$. For $p=2$, the estimate is instead $\norm{U(t)}\lesssim1+\sqrt t$. We therefore use a positive shift $\delta>0$ and the uniformly bounded semigroup $(e^{-\delta t}U(t))_{t\geq0}$. The additional measure condition below makes the corresponding Laplace-calculus identity applicable.

\begin{lemma}\label{lemma:spectrum_U}
Let $1\leq p<\infty$ and let $\mu$ be a finite positive Borel measure on $(0,1]$. If $p\neq2$, assume
\[
    \int_{(0,1]}t^{1-2/p}\,d\mu(t)<\infty, \qquad
    \mathcal M\mu\in\EE_{\gamma_p-1}.
\]
If $p=2$, assume that, for some $\delta>0$,
\[
    \int_{(0,1]}t^{-\delta}\,d\mu(t)<\infty, \qquad
    \mathcal M\mu\in\EE_{\delta-1}.
\]
Then $\BB_\mu$ on $A^p$ satisfies:
\begin{enumerate}[(i)]
    \item $\displaystyle\sigma(\BB_\mu)=\{\mathcal M\mu(z):\Real z\geq2-2/p\}
    \cup\{\mathcal M\mu(1),\mu(\{1\})\}$.
    \item $\displaystyle\sigma_{ess}(\BB_\mu)=\{\mathcal M\mu(z):\Real z=2-2/p\}
    \cup\{\mu(\{1\})\}$.
    \item If $\lambda\notin\sigma_{ess}(\BB_\mu)$, then
    \[
        \operatorname{Ind}(\lambda I-\BB_\mu)
        =\sum_{\substack{\Real z>2-2/p\\\mathcal M\mu(z)=\lambda}}
        \operatorname{ord}_z(\mathcal M\mu-\lambda) < \infty.
    \]
    The operator $\lambda I-\BB_\mu$ is surjective unless $p>2$ and $\lambda=\mathcal M\mu(1)$; in that case its range has codimension one. Its kernel dimension is the displayed sum plus this range codimension.
    \item The point spectrum satisfies
    \begin{align*}
        \sigma_{point}(\BB_\mu)
        &\subseteq\{\mathcal M\mu(z):\Real z>2-2/p\}\cup\{\mathcal M\mu(1),\mathcal M\mu(2-2/p),\mu(\{1\})\},\\
        \sigma_{point}(\BB_\mu)
        &\supseteq\{\mathcal M\mu(z):\Real z>2-2/p\}\cup\{\mathcal M\mu(1)\}.
    \end{align*}
\end{enumerate}
\end{lemma}
\begin{proof}
The assumptions, \Cref{lemma:B_equals_U} and \eqref{eq:average_calculus} give
\[
    \BB_\mu = U_\mu=\int_{[0,\infty)}U(t)\,d\nu(t)=\LL\nu(-B_{A^p}),
\]
using the shift $\gamma_p$ for $p\neq2$ and $\delta$ for $p=2$. Apply \eqref{eq:SMT} and \Cref{lemma:spectrum_U_Bergman} to obtain (i) and (ii). The value $\mathcal M\mu(1)$ comes from the constant eigenfunction. It is already in the half-plane image for $p\leq2$, so retaining it in all cases avoids a separate formula.

For (iii), away from the isolated point $0$ when $p>2$, the zero set in the factorization is finite: a finite accumulation point would contradict holomorphy, while accumulation at an endpoint would put $\lambda$ in the essential spectrum. Thus, the factorization and \eqref{eq:index_FC}, as in \Cref{lemma:spectrum_T}, give the sum over the open half-plane. The only additional issue is the isolated eigenvalue $0$ of $B_{A^p}$ when $p>2$. In that case,
\[
    A^p=\C\oplus\left\{f\in A^p:\int_0^1f(\xi)\,d\xi=0\right\}
\]
is an invariant decomposition for $U(t)$. This follows by applying $J$ to the invariant decomposition $D^p=\C\oplus\{f:f(1)=0\}$. On the constant summand, $\BB_\mu$ is multiplication by $\mathcal M\mu(1)$. On the other summand, the factorization gives surjectivity for $\lambda$ outside the essential spectrum. Thus, the constant summand contributes one to both the kernel dimension and range codimension precisely when $\lambda=\mathcal M\mu(1)$, and contributes zero to the index.

Finally, \eqref{eq:SMT_point} gives (iv). For $p<2$, the finite exceptional value is $\mathcal M\mu(2-2/p)$; for $p>2$, it is $\mathcal M\mu(1)$, already an eigenvalue. With the positive shift used for $p=2$, only the value at infinity can be additional. The stated inclusion covers all three cases.
\end{proof}

\begin{lemma}\label{lemma:duality_Bloch}
Let $\mu$ be a finite positive Borel measure on $(0,1]$ such that $\int_{(0,1]}t^{-1}\,d\mu(t)<\infty$. Then $\Ha_\mu$ is bounded on $\Bloch$ and
\[
    \langle\Ha_\mu f,g\rangle=\langle f,\BB_\mu g\rangle,
    \quad f\in\Bloch,\ g\in A^1.
\]
\end{lemma}
\begin{proof}
Boundedness of $\Ha_\mu$ follows from \cite[Theorem 5.1]{GalanopoulosPapadimitrakis2006}, and boundedness of $\BB_\mu$ follows from \Cref{lemma:B_equals_U}. Let $f(z)=\sum a_kz^k\in\Bloch$ and let $g(z)=\sum_{n=0}^d b_nz^n$ be a polynomial. Using the coefficient representation of the Bloch--Bergman pairing,
\begin{align*}
    \langle\Ha_\mu f,g\rangle
    &=\sum_{n=0}^d\frac{b_n}{n+1}\sum_{k=0}^n c_{n,k}\overline{a_k}\\
    &=\sum_{k=0}^d\frac{\overline{a_k}}{k+1}
    \left((k+1)\sum_{n=k}^d\frac{c_{n,k}}{n+1}b_n\right)
    =\langle f,\BB_\mu g\rangle.
\end{align*}
Density of polynomials in $A^1$ extends this equality to all $g\in A^1$.
\end{proof}

As in the BMOA case, anti-linear duality and the fact that the real matrix $\BB_\mu$ commutes with coefficientwise conjugation give
\[
    \sigma(\Ha_\mu)=\sigma(\BB_\mu)
    , \qquad
    \sigma_{ess}(\Ha_\mu)=\sigma_{ess}(\BB_\mu)
    .
\]
For $p=1$, \Cref{lemma:spectrum_U} gives surjectivity outside the essential spectrum. The kernel and range identities for adjoints therefore prove the injectivity and range-codimension statement in \Cref{th:spectral_result_Hausdorff_Bloch}.

\section{Examples: the spectrum of some Hausdorff matrices on BMOA and \texorpdfstring{$\Bloch$}{B}}\label{sec:examples}

We conclude with two families of Hausdorff matrices considered in \cite{Rhoades1971}. The measures below are absolutely continuous, so they have no atom at $1$.

\subsection{Gamma operators}
Let $\alpha>0$ and $s>1$. Consider the measure
\[
    d\mu_{\alpha,s}(t)=\frac{s^\alpha}{\Gamma(\alpha)}
    t^{s-1}\bigl(\log(1/t)\bigr)^{\alpha-1}\,dt, \quad 0<t<1.
\]
Its moments are $\mu_n=(s/(n+s))^\alpha$. The associated Hausdorff operator is the Gamma operator $\Gamma_s^\alpha=\Ha_{\mu_{\alpha,s}}$. Since
\[
    \int_{(0,1]}t^{-1}\,d\mu_{\alpha,s}(t)
    =\left(\frac{s}{s-1}\right)^\alpha<\infty,
\]
it is bounded on BMOA and $\Bloch$. Its Mellin transform is
\begin{align*}
    \mathcal M\mu_{\alpha,s}(z)
    &=\frac{s^\alpha}{\Gamma(\alpha)}\int_0^1
    t^{s+z-2}\bigl(\log(1/t)\bigr)^{\alpha-1}\,dt\\
    &=\frac{s^\alpha}{\Gamma(\alpha)}\int_0^\infty
    e^{-(s+z-1)t}t^{\alpha-1}\,dt
    =\left(\frac{s}{s+z-1}\right)^\alpha, \quad \Real z>1-s.
\end{align*}
We use the branch determined by the principal logarithm of $s+z-1$. This function is holomorphic on $\C\setminus(-\infty,1-s]$, is holomorphic at $0$, and is bounded on every sector $S_\omega$ with $\omega<\pi$. It is $O(|z|^{-\alpha})$ at infinity in such a sector. Hence, $\mathcal M\mu_{\alpha,s}\in\EE$.

By \Cref{th:spectral_result_Hausdorff_BMOA,th:spectral_result_Hausdorff_Bloch}, on both BMOA and $\Bloch$ we have
\begin{align*}
    \sigma(\Gamma_s^\alpha)
    &=\left\{\left(\frac{s}{s+z-1}\right)^\alpha:\Real z\geq0\right\}\cup\{0\},\\
    \sigma_{ess}(\Gamma_s^\alpha)
    &=\left\{\left(\frac{s}{s+z-1}\right)^\alpha:\Real z=0\right\}\cup\{0\}.
\end{align*}
For $\lambda\notin\sigma_{ess}(\Gamma_s^\alpha)$, the operator $\lambda I-\Gamma_s^\alpha$ is injective and
\[
    \codim\ran(\lambda I-\Gamma_s^\alpha)
    =\left|\left\{z\in\C:\Real z>0,\ \left(\frac{s}{s+z-1}\right)^\alpha=\lambda\right\}\right|.
\]
Here every zero is simple, since the derivative of the Mellin transform does not vanish.

\subsection{Generalized Cesàro operators}
Let $\alpha>0$ and $s>1$. Consider
\[
    d\tilde\mu_{\alpha,s}(t)=\frac{\Gamma(\alpha+s)}{\Gamma(\alpha)\Gamma(s)}
    t^{s-1}(1-t)^{\alpha-1}\,dt, \quad 0<t<1.
\]
Its moments are
\[
    \tilde\mu_n=\frac{\Gamma(\alpha+s)\Gamma(n+s)}
    {\Gamma(n+s+\alpha)\Gamma(s)}.
\]
The generalized Cesàro operator is $\CC_s^\alpha=\Ha_{\tilde\mu_{\alpha,s}}$. It is bounded on BMOA and $\Bloch$, since
\[
    \int_{(0,1]}t^{-1}\,d\tilde\mu_{\alpha,s}(t)
    =\frac{\alpha+s-1}{s-1}<\infty.
\]
Its Mellin transform is
\begin{align*}
    \mathcal M\tilde\mu_{\alpha,s}(z)
    &=\frac{\Gamma(\alpha+s)}{\Gamma(\alpha)\Gamma(s)}
    \int_0^1 t^{s+z-2}(1-t)^{\alpha-1}\,dt\\
    &=\frac{\Gamma(\alpha+s)\Gamma(s+z-1)}
    {\Gamma(s+z-1+\alpha)\Gamma(s)}, \quad \Real z>1-s.
\end{align*}
It extends holomorphically to $\C\setminus(-\infty,1-s]$. The standard Gamma-ratio estimate
\[
    \frac{\Gamma(z+\alpha)}{\Gamma(z)}
    =z^\alpha\left(1+\frac{\alpha(\alpha-1)}{2z}+O(|z|^{-2})\right)
\]
holds as $|z|\to\infty$ in sectors bounded away from the negative real axis, see \cite{ErdelyiTricomi1951}. Thus, $\mathcal M\tilde\mu_{\alpha,s}(z)=O(|z|^{-\alpha})$ on $S_\omega$, $\omega<\pi$. Since it is also holomorphic at $0$, it belongs to $\EE$.

Consequently, on both BMOA and $\Bloch$,
\begin{align*}
    \sigma(\CC_s^\alpha)
    &=\left\{\frac{\Gamma(\alpha+s)\Gamma(s+z-1)}
    {\Gamma(s+z-1+\alpha)\Gamma(s)}:\Real z\geq0\right\}\cup\{0\},\\
    \sigma_{ess}(\CC_s^\alpha)
    &=\left\{\frac{\Gamma(\alpha+s)\Gamma(s+z-1)}
    {\Gamma(s+z-1+\alpha)\Gamma(s)}:\Real z=0\right\}\cup\{0\}.
\end{align*}
For $\lambda\notin\sigma_{ess}(\CC_s^\alpha)$, the operator $\lambda I-\CC_s^\alpha$ is injective and
\[
    \codim\ran(\lambda I-\CC_s^\alpha)
    =\sum_{\substack{\Real z>0\\\mathcal M\tilde\mu_{\alpha,s}(z)=\lambda}}
    \operatorname{ord}_z(\mathcal M\tilde\mu_{\alpha,s}-\lambda).
\]

\section{Acknowledgements}
This work is the result of a research stay undertaken by one of the authors, Alejandro Mahillo, in Thessaloniki during the summer of 2024. He would like to thank Professors Siskakis and Galanopoulos for their hospitality during that period, as well as the entire Department of Mathematics at Aristotle University of Thessaloniki.

\appendix
\section{An auxiliary upper bound on \texorpdfstring{$D^p$}{Dp}}

The estimate below is the special case of the weighted-composition estimates in the proof of \cite[Theorem 6.3]{AbadiasOliva-Maza2024} needed in \Cref{lemma:range_A_Dirichlet}. We give a short verification in our notation. Recall that $\phi_t$ is defined in \eqref{eq:semigroups}.

\begin{lemma}\label{lemma:dirichlet_growth_bound}
Let $1\leq p<\infty$, $m\in\N_0$, and $\lambda\in\C$ satisfy $m+\Real\lambda>2/p$. If $f\in D^p$ has a zero of order at least $m$ at $0$, then
\[
    \norm*{\phi_t\left(\frac{1-\phi_t}{1-(\cdot)}\right)^\lambda(f\circ\phi_t)}_{D^p}
    \lesssim e^{(\Real\lambda+1-2/p)t}\norm{f}_{D^p}, \quad t\geq0.
\]
The implicit constant may depend on $p,m,\lambda$, but not on $t$ or $f$.
\end{lemma}
\begin{proof}
Put $f_m(z)=f(z)/z^m$, with the removable singularity filled in. Boundedness of the backward shift on $D^p$ gives $\norm{f_m}_{D^p}\lesssim\norm{f}_{D^p}$. Write
\[
    g_t(z)=\left(\frac{\phi_t(z)}{z}\right)^{m+1}
    \left(\frac{1-\phi_t(z)}{1-z}\right)^\lambda,
    \qquad h_t=g_t(f_m\circ\phi_t).
\]
The function in the statement is $z^{m+1}h_t$, so it is enough to estimate $\|h_t\|_{D^p}$. We have $h_t(0)=e^{-(m+1)t}f_m(0)$ and
\[
    h_t'=g_t'(f_m\circ\phi_t)+g_t\phi_t'(f_m'\circ\phi_t).
\]
A change of variable gives $\norm{(\phi_t')^{2/p}(h\circ\phi_t)}_{A^p}\leq\norm{h}_{A^p}$. Direct differentiation, using $m+\Real\lambda>2/p$, gives
\[
    \norm{g_t'(\phi_t')^{-2/p}}_\infty
    +\norm{g_t(\phi_t')^{1-2/p}}_\infty
    \lesssim e^{(\Real\lambda+1-2/p)t}.
\]
Indeed, the two expressions are
\[
    \frac{(m+\lambda+1)(1-e^{-t})e^{-(m+1-2/p)t}}
    {((e^{-t}-1)z+1)^{m+\lambda+2-4/p}}, \qquad
    \frac{e^{-(m+2-2/p)t}}
    {((e^{-t}-1)z+1)^{m+\lambda+3-4/p}}.
\]
The real parts of the denominator exponents are positive, and $|1+(e^{-t}-1)z|\geq e^{-t}$. Combining these estimates with $\norm{f_m}_{A^p}\lesssim\norm{f_m}_{D^p}$ proves the result.
\end{proof}

\printbibliography
\end{document}